\documentclass[12pt,a4]{amsart}
\usepackage[a4paper, left=28mm, right=28mm, top=28mm, bottom=34mm]{geometry}
\usepackage[all]{xy}
\usepackage{amsmath}
\usepackage{amssymb}
\usepackage{amsthm}
\usepackage{mathrsfs}
\theoremstyle{definition}
\newtheorem{thm}{Theorem}[section]
\newtheorem{lem}[thm]{Lemma}
\newtheorem{cor}[thm]{Corollary}
\newtheorem{prop}[thm]{Proposition}

\theoremstyle{definition}

\newtheorem{rem}[thm]{Remark}

\numberwithin{equation}{section}

\def\F{{\mathbb F}}

\def\Q{{\mathbb Q}}

\def\Z{{\mathbb Z}}
\def\C{{\mathbb C}}

\def\PP{{\mathbb P}}

\def\Aut{\mathop{\mathrm{Aut}}\nolimits}

\def\Gal{\mathop{\mathrm{Gal}}\nolimits}

\def\Jac{\mathop{\mathrm{Jac}}\nolimits}
\def\Pic{\mathop{\mathrm{Pic}}\nolimits}

\def\id{\mathop{\mathrm{id}}\nolimits}
\def\GL{\mathop{\mathrm{GL}}\nolimits}

\def\ord{\mathop{\mathrm{ord}}}

\begin{document}

\title[Mod 4 Galois representation associated with the Fermat quartic]
{Explicit calculation of the mod 4 Galois representation associated with the Fermat quartic}

\author{Yasuhiro Ishitsuka}
\address{Department of Mathematics, Faculty of Science, Kyoto University, Kyoto 606--8502, Japan}
\email{yasu-ishi@math.kyoto-u.ac.jp}

\author{Tetsushi Ito}
\address{Department of Mathematics, Faculty of Science, Kyoto University, Kyoto 606--8502, Japan}
\address{Mathematical Science Team, RIKEN Center for Advanced Intelligence Project (AIP),1--4--1 Nihonbashi, Chuo-ku, Tokyo 103--0027, Japan}
\email{tetsushi@math.kyoto-u.ac.jp}

\author{Tatsuya Ohshita}
\address{Department of Mathematics, Faculty of Science and Technology, Keio University, 3--14--1 Hiyoshi, Kohoku-ku, Yokohama-shi, Kanagawa 223--8522, Japan}
\email{ohshita.tatsuya.nz@ehime-u.ac.jp}

\date{September 29, 2019}

\begin{abstract}
We use explicit methods to study
the $4$-torsion points on the Jacobian variety of the Fermat quartic.
With the aid of computer algebra systems,
we explicitly give a basis of the group of $4$-torsion points.
We calculate the Galois action,
and show that the image of the mod $4$ Galois representation
is isomorphic to the dihedral group of order $8$.
As applications,
we calculate the Mordell-Weil group of the Jacobian variety of
the Fermat quartic over each subfield of the $8$-th cyclotomic field.
We determine all of the points on the Fermat quartic defined
over quadratic extensions of the $8$-th cyclotomic field.
Thus we complete Faddeev's work in 1960.
\end{abstract}

\maketitle

\section{Introduction}

In this paper, we use explicit methods to study
rational points and divisors on the \textit{Fermat quartic} $F_4 \subset \PP^2$
defined by the homogeneous equation
\[ X^4 + Y^4 = Z^4. \]
We study the $4$-torsion points $\Jac(F_4)[4]$
on the Jacobian variety $\Jac(F_4)$.
Then we give applications to the Mordell-Weil group of $\Jac(F_4)$
and the rational points on $F_4$.

In the 17th century, Fermat proved that
all of the $\Q$-rational points on $F_4$ satisfy $XYZ = 0$.
(The points on $F_4$ satisfying $XYZ = 0$ are called the \textit{cusps}.)
Since the work of Fermat, the arithmetic of $F_4$ has been studied
by many mathematicians because, together with the Fermat curves of higher degree,
it provides a nice testing ground of more general theories.

In 1960, Faddeev studied the arithmetic of $\Jac(F_4)$ using
methods from algebraic geometry \cite{Faddeev}.
He constructed an isogeny of degree $8$ defined over $\Q$
from $\Jac(F_4)$ to the product of three elliptic curves.
Since the Mordell-Weil groups of these elliptic curves over $\Q$
are finite of $2$-power order,
it follows that the Mordell-Weil group $\Jac(F_4)(\Q)$
is also finite of $2$-power order.

It is an interesting problem to study the $2$-power torsion points.
But the situation was not clear in the literature.
In \cite{Faddeev}, Faddeev claimed (without proof) that
$\Jac(F_4)(\Q)$ has order $32$,
and the Mordell-Weil group of $\Jac(F_4)$
over the $8$-th cyclotomic field $\Q(\zeta_8)$ is finite.
(But he did not mention the structure of the Mordell-Weil group.)
He alluded that he could determine all of the points on $F_4$
defined over quadratic extensions of $\Q(\zeta_8)$;
see \cite[p.1150]{Faddeev}.
As far as the authors of this paper know,
the precise statements of Faddeev's claims (and proofs of them)
did not appear in the literature.
Nevertheless, Faddeev's results and claims were cited several times.
Kenku used them to study $2$-power torsion points
on elliptic curves over quadratic fields \cite{Kenku}.
Klassen revisited Faddeev's results in his thesis \cite{Klassen:Thesis}.
Schaefer-Klassen used the finiteness of $\Jac(F_4)(\Q(\zeta_8))$
to study the torsion packet on $F_4$ \cite[Section 6]{KlassenSchaefer}.

In this paper,
we use explicit methods to study the $4$-torsion points
$\Jac(F_4)[4]$ on $\Jac(F_4)$.
We give an explicit description of
the action of the absolute Galois group
$\Gal(\overline{\Q}/\Q)$ on $\Jac(F_4)[4]$.
As applications of our results,
we give precise statements alluded by Faddeev, and prove them.
Thus we fill the gap in the literature.

Here is a summary of the results we obtain in this paper:
\begin{enumerate}
\item (Theorem \ref{MainTheorem1})\ We give $6$ divisors of degree $0$ on $F_4$
which give a basis of $\Jac(F_4)[4]$ as a $\Z/4\Z$-module.
Five of them are supported on the cusps, but one of them is not.

\item (Theorem \ref{MainTheorem2})\ We calculate
the mod $4$ Galois representation:
\[ \rho_4 \colon \Gal(\overline{\Q}/\Q) \to \Aut(\Jac(F_4)[4]) \cong \GL_6(\Z/4\Z). \]
The image of $\rho_4$ is isomorphic to
the dihedral group of order $8$,
and the kernel of $\rho_4$ corresponds to the number field $\Q(2^{1/4},\zeta_8)$,
which is a quadratic extension of $\Q(\zeta_8)$.
(We also calculate the Weil pairing.
We determine the image of $\rho_4$
inside the symplectic similitude group $\mathrm{GSp}_6(\Z/4\Z)$.
See Theorem \ref{Theorem:WeilPairing} and Corollary \ref{Corollary:WeilPairingGSp(6)}.)

\item (Theorem \ref{MainTheorem3})\ We calculate
the Mordell-Weil group $\Jac(F_4)(\Q(\zeta_8))$.
We show $\Jac(F_4)(\Q(\zeta_8))$ is isomorphic to
$(\Z/4\Z)^{\oplus 5} \oplus \Z/2\Z$ generated by
divisors supported on the cusps.
(We also calculate the Mordell-Weil group of $\Jac(F_4)$
and its generators over each subfield of $\Q(\zeta_8)$
(i.e., over $\Q$, $\Q(\sqrt{-1})$, $\Q(\sqrt{2})$, and $\Q(\sqrt{-2})$).
See Appendix \ref{Appendix:MordellWeilGroupSubfields}.)

\item (Theorem \ref{MainTheorem4})\
We determine all of the points on $F_4$
defined over quadratic extensions of $\Q(\zeta_8)$.
We list all of them.
It turns out that $F_4$ has $188$ such points in total.
Except for the $12$ cusps, none of them is defined over $\Q(\zeta_8)$.
There exist $48$ points defined over $\Q(2^{1/4},\zeta_8)$,
$32$ points defined over $\Q(\zeta_3, \zeta_8)$,
and $96$ points defined over $\Q(\sqrt{-7}, \zeta_8)$.
\end{enumerate}

Note that we take advantage of using results and techniques
which were not available in 1960's.
Especially, our proof depends on Rohrlich's results on the subgroup of
the Mordell-Weil group generated by
divisors supported on the cusps \cite{Rohrlich}.
We obtained and confirmed key computational results with
the aid of computer algebra systems Maxima, Sage, and Singular
\cite{Maxima}, \cite{Sage}, \cite{Singular}, \cite{Singular:DivisorsLib}.

The organization of this paper is as follows.
In Section \ref{Section:Notation},
we introduce necessary notation on the Galois group and
divisors on the Fermat quartic $F_4$.
In Section \ref{Section:Rohrlich},
we recall Rohrlich's results on the cusps.
In Section \ref{Section:Mod4},
we give a basis of the $4$-torsion points $\Jac(F_4)[4]$.
In Section \ref{Section:GaloisAction}, we calculate the Galois action.
Using these results, the calculation of
the Mordell-Weil group of $\Jac(F_4)$ over $\Q(\zeta_8)$
becomes an exercise in linear algebra.
We calculate it in Section \ref{Section:MordellWeilGroupQZeta8}.
In Section \ref{Section:QuadraticPoints},
we determine all of the points on $F_4$
defined over quadratic extensions of $\Q(\zeta_8)$.
This paper has $5$ appendices.
In Appendix \ref{Appendix:MordellWeilGroupSubfields},
we calculate the Mordell-Weil group of $\Jac(F_4)$
over each subfield of $\Q(\zeta_8)$.
In Appendix \ref{Appendix:WeilPairing},
we calculate the Weil pairing on $\Jac(F_4)[4]$.
In Appendix \ref{Appendix:Automorphisms},
we calculate the action of the automorphism group of $F_4$
on the $4$-torsion points $\Jac(F_4)[4]$.
In Appendix \ref{Appendix:Experimental},
we briefly explain the authors' experimental methods to find a $4$-torsion point
which does not belong to the subgroup generated by divisors
supported on the cusps.
Finally, in Appendix \ref{Appendix:MethodsCalculation},
we give several remarks on the methods of calculation.

The initial motivation of this work was to
understand Faddeev's results and claims in \cite{Faddeev},
and to apply them to study linear and symmetric determinantal representations of
the Fermat quartic $F_4$ over $\Q$.
Such applications will appear elsewhere;
see \cite{IshitsukaItoOhshita:KleinFermatQuartic}
for the summary of our results.

\section{Notation}
\label{Section:Notation}

In this section, we shall introduce necessary notation on
the Galois group and divisors which will be used in the rest of
this paper.

We fix an embedding of an algebraic closure $\overline{\Q}$ of $\Q$
into the field $\C$ of complex numbers.
We also put
$\zeta_n := \exp(2 \pi i/n)$.
Hence we have $(\zeta_{mn})^m = \zeta_n$ for any $n,m \geq 1$.
For a positive integer $a \geq 1$, we put $\sqrt{-a} := \sqrt{a} \zeta_4$.
The element $2^{1/4}$ denotes
a unique positive real number whose fourth power is $2$.
The following relation holds:
$\zeta_8 + \zeta^7_8 = (2^{1/4})^2 = \sqrt{2}$.
The field $\Q(2^{1/4},\zeta_8)$
is a Galois extension of $\Q$ of degree $8$.
The Galois group $\Gal(\Q(2^{1/4},\zeta_8)/\Q)$
is generated by the following two automorphisms $\sigma,\tau$:
$\sigma$ is an element of order $4$ satisfying
$\sigma(\zeta_8) = \zeta_8^5 = -\zeta_8$ and
$\sigma(2^{1/4}) = 2^{1/4} \zeta_4$,
and $\tau$ is an element of order $2$ satisfying
$\tau(\zeta_8)   = \zeta_8^7$ and
$\tau(2^{1/4})   = 2^{1/4}$.
These elements satisfy $\tau \sigma \tau = \sigma^{3}$.
We see that $\Gal(\Q(2^{1/4},\zeta_8)/\Q)$ is isomorphic to
the dihedral group with $8$ elements.

The \textit{Fermat quartic} is a smooth projective curve over $\Q$ defined by
\[ F_4 := \{\, [X:Y:Z] \in \PP^2 \mid X^4 + Y^4 = Z^4 \,\}. \]
Its Jacobian variety $\Jac(F_4)$ is an abelian variety of dimension $3$ over $\Q$.
The group of $4$-torsion points
\[ \Jac(F_4)[4] := \{\, \alpha \in \Jac(F_4)(\overline{\Q}) \mid [4] \alpha = 0 \,\} \]
is a free $\Z/4\Z$-module of rank $6$
with an action of $\Gal(\overline{\Q}/\Q)$,
where $[4]$ denotes the multiplication-by-4 isogeny on $\Jac(F_4)$.
Since $F_4$ has a $\Q$-rational point (such as $[1 : 0 : 1]$),
for any extension $k/\Q$,
the degree $0$ part of the Picard group of $F_4 \otimes_{\Q} k$
is canonically identified with the group of $k$-rational points on $\Jac(F_4)$:
\[ \Pic^0(F_4 \otimes_{\Q} k) \cong \Jac(F_4)(k). \]
(See \cite[Chapter 8, Proposition 4]{BoschLuetkebohmertRaynaud},
\cite[Section 5.7.1]{Poonen:RationalPoints}.)

The Fermat quartic $F_4$ has $12$ points satisfying $XYZ = 0$:
\begin{align*}
A_i &:= [0 : \zeta_4^i : 1], &
B_i &:= [\zeta_4^i : 0 : 1], &
C_i &:= [\zeta_8 \zeta_4^i : 1 : 0]
\end{align*}
for $0 \leq i \leq 3$.
These points are called \textit{cusps}.
All of the cusps on $F_4$ are defined over $\Q(\zeta_8)$.
We introduce the following points defined over $\Q(2^{1/4},\zeta_8)$:
\begin{align*}
P_1 &:= [ 2^{1/4} \zeta_4 : \zeta_8 : 1 ], &
P_2 &:= [ \zeta_8 : 2^{1/4} \zeta_4 : 1 ], &
P_3 &:= [ 2^{-1/4} : 2^{-1/4} : 1 ].
\end{align*}

We embed $F_4$ into $\Jac(F_4)$ by taking $B_0$ as a base point.
For each $0 \leq i \leq 3$, we define
$\alpha_i, \beta_i, \gamma_i \in \Jac(F_4)(\overline{\Q})$ by
\begin{align*}
\alpha_i &:= [A_i - B_0], &
\beta_i  &:= [B_i - B_0], &
\gamma_i &:= [C_i - B_0],
\end{align*}
where the linear equivalence class of a divisor $D$ is denoted by $[D]$.
We define
$e_1, e_2, e_3, e_4, e_5 \in \Jac(F_4)(\overline{\Q})$ by
\begin{align*}
e_1 &:= \alpha_1, &
e_2 &:= \alpha_2, &
e_3 &:= \beta_1, &
e_4 &:= \beta_2, &
e_5 &:= \gamma_1.
\end{align*}
We also define $e_6, e'_6 \in \Jac(F_4)(\overline{\Q})$ by
\begin{align*}
e_6 &:= [A_1 + A_2 + B_1 + B_2 + C_1 + C_2 - 6 B_0], \\
e'_6 &:= [P_1 + P_2 + P_3 - 3 B_0].
\end{align*}

It turns out that the divisor $P_1 + P_2 + P_3$,
defined over $\Q(2^{1/4},\zeta_8)$, plays an important role
in the arithmetic of the Fermat quartic $F_4$.
This divisor does not seem to receive special attention before.

\section{Rohrlich's results on divisor classes of the cusps}
\label{Section:Rohrlich}

Let
$\mathscr{C} \subset \Jac(F_4)(\overline{\Q})$
be the subgroup generated by
$\alpha_i, \beta_i, \gamma_i \ (0 \leq i \leq 3)$.
(The group $\mathscr{C}$ is denoted by $\mathscr{D}^{\infty}/\mathscr{F}^{\infty}$
in \cite{Rohrlich}.)
Rohrlich determined all of the relations between $\alpha_i, \beta_i, \gamma_i$,
and calculated the group $\mathscr{C}$.

Here is a summary of Rohrlich's results:

\begin{prop}[Rohrlich {\cite[p.117, Corollary 1]{Rohrlich}}]
\label{Proposition:Rohrlich}
The group $\mathscr{C}$ is isomorphic to the abelian group
generated by  $\alpha_i, \beta_i, \gamma_i$ ($0 \leq i \leq 3$)
with the following relations:
\begin{align*}
0 &= 4\alpha_i = 4\beta_i = 4\gamma_i, \\
0 &= \alpha_0 + \alpha_1 + \alpha_2 + \alpha_3
   = \beta_1 + \beta_2 + \beta_3
   = \gamma_0 + \gamma_1 + \gamma_2 + \gamma_3, \\
0 &= \alpha_1 + \beta_1 + 2(\alpha_2 + \beta_2) +3( \alpha_3 + \beta_3), \\
0 &= \beta_1 + \gamma_1 + 2(\beta_2 + \gamma_2) + 3(\beta_3 + \gamma_3), \\
0 &= 2 (\alpha_1 + \beta_1 + \gamma_1 + \alpha_2 + \beta_2 + \gamma_2).
\end{align*}
\end{prop}

\begin{cor}
\label{Corollary:Rohrlich}
\begin{enumerate}
\item $\mathscr{C}$ is a finite abelian group killed by $4$.
\item $e_6$ is killed by $2$.
\item The following homomorphism is an isomorphism:
\begin{align*}
  (\Z/4\Z)^{\oplus 5} \oplus (\Z/2\Z) &\overset{\cong}{\longrightarrow} \mathscr{C}, \\
  (c_1,c_2,c_3,c_4,c_5,c_6) &\mapsto \sum_{i=1}^{5} c_i e_i + c_6 e_6.
\end{align*}
\end{enumerate}
\end{cor}

For later use, we note that the following equalities are satisfied:
\begin{align*}
\alpha_0 &= 2e_1 + e_2 + 2e_3 + e_4, &
\alpha_3 &= e_1 + 2e_2 + 2e_3 + 3 e_4, \\
\beta_0 &= 0, &
\beta_3 &= 3 e_3 + 3 e_4, \\
\gamma_0 &= 3 e_1 + 3 e_2 + e_3 + e_5 + e_6, &
\gamma_2 &= 3 e_1 + 3 e_2 + 3 e_3 + 3 e_4 + 3 e_5 + e_6, \\
\gamma_3 &= 2 e_1 + 2 e_2 + e_4 + 3 e_5.
\end{align*}

\section{Explicit determination of the $4$-torsion points}
\label{Section:Mod4}

By Rohrlich's results,
the subgroup $\mathscr{C} \subset \Jac(F_4)[4]$ has index $2$;
see Corollary \ref{Corollary:Rohrlich}.
Hence we need only to find a $4$-torsion point on $\Jac(F_4)$
which does not belong to $\mathscr{C}$.
Currently, there seems no practical algorithm to calculate such a point.
Fortunately, after a number of trials and errors,
the authors found that $e'_6$ does the job.
(See also Appendix \ref{Appendix:Experimental}.)

\begin{prop}
\label{Proposition:KeyDivisor}
The element $e'_6$ satisfies the following equality in $\Jac(F_4)(\overline{\Q})$:
\[ 2 e'_6 = 2e_2 + 2e_4 + e_6. \]
\end{prop}

\begin{proof}
It is enough to show that
\[
2 P_1 + 2 P_2 + 2 P_3
- A_1 - 3 A_2 + 4 B_0 - B_1 - 3 B_2 - C_1 - C_2
\]
is linearly equivalent to $0$.
We shall give a rational function $f$ whose divisor coincides
with the above divisor.
The field $\Q(2^{1/4}, \zeta_8)$
is generated by an element $\delta$ such that its minimal polynomial over $\Q$ is
\[ X^8 - 4X^6 + 8X^4 - 4X^2 + 1, \]
and the following equalities are satisfied:
\begin{align*}
2 \cdot \delta^2 &= (2 - \sqrt{2})(1 + \zeta_4), \\
3 \cdot \zeta_8 &= 2 \delta^6 - 7 \delta^4 + 11 \delta^2 - 1, \\
3 \cdot 2^{1/4} &= \delta^7 - 5 \delta^5 + 10 \delta^3 - 8 \delta.
\end{align*}
We define the elements $c_1,c_2,c_3,c_4,c_5$ by
\begin{align*}
c_1 &:= \delta^6 - 2 \delta^4 + \delta^2 + 7, \\
c_2 &:= 2 \delta^6 - 10 \delta^4 + 20 \delta^2 - 13, \\
c_3 &:= 22 \delta^7 + 9 \delta^6 - 86 \delta^5 - 36 \delta^4 + 166 \delta^3 + 63 \delta^2 - 86 \delta - 24, \\
c_4 &:= 10 \delta^7 - 2 \delta^6 - 26 \delta^5 + 7 \delta^4 + 22 \delta^3 - 20 \delta^2 + 46 \delta - 14, \\
c_5 &:= 22 \delta^6 - 77 \delta^4 + 154 \delta^2 - 44.
\end{align*}
Then, it can be checked (with the aid of computer algebra systems)
that the following rational function $f$ satisfies the required conditions:
\begin{align*}
f &:= \frac{g_2}{g_1}, \\
g_1 &:=  3 (X^{3} + Y^{3} + Z^{3}) + c_1 (X^{2}Y + X^{2}Z + XY^{2} + XYZ + Y^{2}Z - Z^3) \\
    &\quad -c_2(XYZ + XZ^{2} + YZ^{2} + Z^{3}), \\
g_2 &:= 33 (X^3 + XY^2 + XYZ - XZ^2 - Y^2 Z - YZ^2) \\
    &\quad + c_3 (-X^2 Y + XYZ) + c_4 (X^2Z + XYZ - XZ^2 - YZ^2) \\
    &\quad + c_5 (XZ^2 - Z^3).
\end{align*}
(See also Appendix \ref{Appendix:MethodsCalculation}.)
\end{proof}

By Corollary \ref{Corollary:Rohrlich},
the element $2e_2 + 2e_4 + e_6$ is killed by $2$,
but it is not divisible by $2$ inside $\mathscr{C}$.
From these calculations,
it is straightforward to give a basis of $\Jac(F_4)[4]$.

\begin{thm}
\label{MainTheorem1}
The group of $4$-torsion points $\Jac(F_4)[4]$
is a free $\Z/4\Z$-module of rank $6$
with basis $e_1$, $e_2$, $e_3$, $e_4$, $e_5$, $e'_6$.
In other words, the following homomorphism is an isomorphism:
\begin{align*}
  (\Z/4\Z)^{\oplus 6} &\overset{\cong}{\longrightarrow} \Jac(F_4)[4], \\
  (c_1,c_2,c_3,c_4,c_5,c'_6) &\mapsto \sum_{i=1}^{5} c_i e_i + c'_6 e'_6.
\end{align*}
In particular, all of the $4$-torsion points on $\Jac(F_4)$
are defined over $\Q(2^{1/4}, \zeta_8)$.
\end{thm}

\begin{rem}
\label{Remark:ModularCurve}
The element $e'_6$ was found experimentally by the authors;
see Appendix \ref{Appendix:Experimental}.
The authors were unaware of any theoretical meaning of this element.
It would be interesting to look for a ``natural'' or ``moduli theoretic''
proof of Proposition \ref{Proposition:KeyDivisor}.
(Note that the Fermat quartic $F_4$ is isomorphic to
the modular curve $X_0(64)$ over $\Q$;
see \cite[Proposition 2]{Kenku}, \cite[p.454]{MShimura}, \cite[p.107]{TuYang}.)
\end{rem}

\section{Explicit calculation of the Galois action}
\label{Section:GaloisAction}

By Theorem \ref{MainTheorem1},
we see that the action of $\Gal(\overline{\Q}/\Q)$
on $\Jac(F_4)[4]$ factors through $\Gal(\Q(2^{1/4},\zeta_8)/\Q)$.
Hence we have the mod $4$ Galois representation
\[ \rho_4 \colon \Gal(\Q(2^{1/4},\zeta_8)/\Q) \to \Aut(\Jac(F_4)[4]) \cong \GL_6(\Z/4\Z) \]
with respect to the basis $e_1,e_2,e_3,e_4,e_5,e'_6$.

Our task is to calculate the matrices $\rho_4(\sigma), \rho_4(\tau)$ explicitly.

\begin{thm}
\label{MainTheorem2}
With respect to the basis $e_1,e_2,e_3,e_4,e_5,e'_6 \in \Jac(F_4)[4]$,
the actions of $\sigma,\tau$ on $\Jac(F_4)[4]$
are represented by the following matrices:
\begin{align*}
\rho_4(\sigma) &= \begin{pmatrix}
1 & 0 & 0 & 0 & 2 & 1 \\
0 & 1 & 0 & 0 & 2 & 3 \\
0 & 0 & 1 & 0 & 0 & 3 \\
0 & 0 & 0 & 1 & 1 & 1 \\
0 & 0 & 0 & 0 & 3 & 2 \\
0 & 0 & 0 & 0 & 0 & 3
\end{pmatrix}, &
\rho_4(\tau) &= \begin{pmatrix}
1 & 0 & 0 & 0 & 3 & 0 \\
2 & 1 & 0 & 0 & 1 & 3 \\
2 & 0 & 3 & 0 & 3 & 0 \\
3 & 0 & 3 & 1 & 1 & 3 \\
0 & 0 & 0 & 0 & 3 & 0 \\
0 & 0 & 0 & 0 & 2 & 3
\end{pmatrix}.
\end{align*}
In particular, the mod $4$ Galois representation $\rho_4$ is injective,
and the image of $\rho_4$ is isomorphic to the dihedral group of order $8$.
\end{thm}

\begin{proof}
It is a straightforward exercise (with the aid of computer algebra systems).
We briefly give a summary of our calculations.

The actions of $\sigma,\tau$
on the cusps $A_i,B_i,C_i$ $(0 \leq i \leq 3)$
are calculated as follows:
\[
\begin{array}{|c||c|c|c|c||c|c|c|c||c|c|c|c|}
\hline
\quad & A_0 & A_1 & A_2 & A_3 & B_0 & B_1 & B_2 & B_3 & C_0 & C_1 & C_2 & C_3 \\
\hline
\sigma & A_0 & A_1 & A_2 & A_3 & B_0 & B_1 & B_2 & B_3 & C_2 & C_3 & C_0 & C_1 \\
\hline
\tau   & A_0 & A_3 & A_2 & A_1 & B_0 & B_3 & B_2 & B_1 & C_3 & C_2 & C_1 & C_0 \\
\hline
\end{array}
\]

The actions of $\sigma,\tau$ on $P_1,P_2,P_3$ are calculated
as follows:
\[
\begin{array}{|c||c|c|c|}
\hline
\quad & P_1 & P_2 & P_3 \\
\hline
\sigma & [ -2^{1/4} : -\zeta_8 : 1 ] & [ -\zeta_8 : -2^{1/4} : 1 ] &
[ -2^{-1/4} \zeta_4 : -2^{-1/4} \zeta_4 : 1 ] \\
\hline
\tau   & [ -2^{1/4} \zeta_4 : \zeta_8^7 : 1 ] & [ \zeta_8^7 : -2^{1/4} \zeta_4 : 1 ] & P_3 \\
\hline
\end{array}
\]

By Theorem \ref{MainTheorem1},
$\Jac(F_4)[4]$ is a free $\Z/4\Z$-module with basis
$e_1$, $e_2$, $e_3$, $e_4$, $e_5$, $e'_6$.
Hence we need to calculate the actions of $\sigma,\tau$ on these elements.
\[
\begin{array}{|c||c|c|c|c|}
\hline
\quad & e_1 & \qquad e_2 \qquad & e_3 & \qquad e_4 \qquad \\
\hline
\sigma & e_1 & e_2 & e_3 & e_4 \\
\hline
\tau &
\begin{array}{c}
[A_3 - B_0] = \\ e_1 + 2 e_2 + 2 e_3 + 3 e_4
\end{array}
& e_2 &
\begin{array}{c}
[B_3 - B_0] = \\ 3 e_3 + 3 e_4
\end{array}
& e_4 \\
\hline
\end{array}
\]

\[
\begin{array}{|c||c|c|}
\hline
\quad & e_5 & e'_6 \\
\hline
\sigma &
\begin{array}{c}
[C_3 - B_0] = \\ 2 e_1 + 2 e_2 + e_4 + 3 e_5
\end{array}
&
e_1 + 3 e_2 + 3 e_3 + e_4 + 2 e_5 + 3 e'_6 \\
\hline
\tau   &
\begin{array}{c}
[C_2 - B_0] = \\
3 e_1 + e_2 + 3 e_3 + e_4 + 3 e_5 + 2e'_6
\end{array} &
3 e_2 + 3 e_4 + 3 e'_6 \\
\hline
\end{array}
\]

From these calculations,
we see that the images of $\sigma,\tau$ under the homomorphism $\rho_4$
are given by the above matrices.
\end{proof}

\begin{rem}
The dihedral extension $\Q(2^{1/4},\zeta_8)/\Q$ appeared in
a recent work of Kiming-Rustom \cite{KimingRustom}.
They constructed a $2$-dimensional mod $4$ Galois representation
which factors through $\Gal(\Q(2^{1/4},\zeta_8)/\Q)$,
and showed that it is associated with a cusp form $f$
of weight $36$ and level $1$.
They also showed that it is realized as the Galois action
on $4$-torsion points on the elliptic curve
$E \colon Y^2 = X^3 + X^2 + X + 1$,
which is a modular elliptic curve of conductor $128$
associated with a cusp form $g$ of weight $2$ and level $128$;
the cusp forms $f,g$ are congruent mod $4$.
(See \cite[Theorem 1]{KimingRustom}.)
Since $F_4 \cong X_0(64)$,
it is interesting to study the relation between
the Galois action on $\Jac(F_4)[4]$ and
the $2$-dimensional mod $4$ Galois representation
constructed by Kiming-Rustom.
\end{rem}

\section{Calculation of the Mordell-Weil group over $\Q(\zeta_8)$}
\label{Section:MordellWeilGroupQZeta8}

In this section, as an application of our results
(Theorem \ref{MainTheorem1}, Theorem \ref{MainTheorem2}),
we calculate the Mordell-Weil group of $\Jac(F_4)$ over $\Q(\zeta_8)$.

We briefly recall the isogeny constructed by Faddeev \cite{Faddeev}.
Let $E_1$ (resp.\ $E_3$)
be the smooth projective curve over $\Q$
birational to the affine curve $Y^2 = 1 - X^4$ (resp.\ $Y^2 = 1 + X^4$).
Then, $E_1$ (resp.\ $E_3$) is isomorphic to the elliptic curve
whose Weierstrass equation is $Y^2 = X^3 + 4X$ (resp.\ $Y^2 = X^3 - 4X$).
For any $\lambda$,
the affine curve $C_{\lambda} \colon Y^2 = 1 - \lambda X^4$ is
birational to the affine curve $C'_{\lambda} \colon Y^2 = X^3 + 4 \lambda X$
via the rational map
\[ C_{\lambda} \dashrightarrow C'_{\lambda},\quad (a,b) \mapsto (u,v) := (2 \lambda a^2/(1-b),\,4 \lambda a/(1-b)); \]
the inverse map is given by
\[ C'_{\lambda} \dashrightarrow C_{\lambda},\quad (u,v) \mapsto (2u/v,\,1 - 8 \lambda u/v^2). \]

We put $E_2 := E_1$.
Then we have morphisms
$f_i \colon F_4 \to E_i$ $(1 \leq i \leq 3)$
of degree $2$ by
\begin{align*}
f_1([X : Y : Z]) &= ((Y/Z),\,(X/Z)^2), \\
f_2([X : Y : Z]) &= ((X/Z),\,(Y/Z)^2), \\
f_3([X : Y : Z]) &= ((X/Y),\,(Z/Y)^2).
\end{align*}
These morphisms induce a homomorphism of abelian varieties
\begin{equation}
\label{Jacobian:Isogeny}
\xymatrix{
\Jac(F_4) \ar^-{}[r] & E_1 \times E_2 \times E_3
\qquad
}
\end{equation}
defined over $\Q$. It is an isogeny of degree $8$.

The following result is presumably well-known.

\begin{lem}
\label{Lemma:MordellWeilGroup1}
The Mordell-Weil group $\Jac(F_4)(\Q(\zeta_8))$ is a finite abelian group
whose order is a power of $2$.
\end{lem}

\begin{proof}
Since the degree of the isogeny (\ref{Jacobian:Isogeny}) is a power of $2$,
it is enough to show the same assertion for each
$E_i$ ($1 \leq i \leq 3$).
Over $\Q(\zeta_8)$, the elliptic curves $E_1,E_2,E_3$ are isomorphic to each other.
Hence it is enough to show the assertion for $E_1$.
It should be possible to calculate $E_1(\Q(\zeta_8))$ directly.
But the calculation of the Mordell-Weil group is not an easy problem over a number field of large degree.

Here is a simple proof of this lemma avoiding the calculation of the Mordell-Weil group
over a number field other than $\Q$.
The Mordell-Weil group of $E_1$ over $\Q(\zeta_8)$
is identified with the Mordell-Weil group of the Weil restriction
$\mathrm{Res}_{\Q(\zeta_8)/\Q}(E_1)$,
which is an abelian variety of dimension $4$ over $\Q$.
Since $\Q(\zeta_8)$ is a biquadratic field equal to
the composite of $\Q(\sqrt{-1})$ and $\Q(\sqrt{2})$,
there is an isogeny of $2$-power degree from
$\mathrm{Res}_{\Q(\zeta_8)/\Q}(E_1)$
to the product of the following $4$ elliptic curves over $\Q$:
\begin{align*}
 \pm Y^2 &= X^3 + 4 X, &
 \pm 2Y^2 &= X^3 + 4 X.
\end{align*}
It is an easy exercise to check that the Mordell-Weil groups of
these elliptic curves over $\Q$ are finite abelian groups of $2$-power order.
\end{proof}

Let $\F_3$ (resp.\ $\F_9$) be the finite field with $3$ (resp.\ $9$) elements.

\begin{lem}
\label{Lemma:MordellWeilGroup2}
Let $\widetilde{F}_4$ be the smooth quartic over $\F_3$ defined by $X^4 + Y^4 = Z^4$.
Let $\Jac(\widetilde{F}_4)$ be the Jacobian variety of $\widetilde{F}_4$.
\begin{enumerate}
\item For any prime number $\ell \neq 3$,
every eigenvalue of the Frobenius morphism over $\F_3$
on the $\ell$-adic Tate module
\[ V_{\ell} \Jac(\widetilde{F}_4) :=
   \big( \varprojlim_{n} \Jac(\widetilde{F}_4)[\ell^n] \big) \otimes_{\Z_{\ell}} \Q_{\ell}.
\]
is a square root of $-3$.

\item
For any prime number $\ell \neq 3$,
the Frobenius morphism over $\F_9$ acts on
$V_{\ell} \Jac(\widetilde{F}_4)$
via the scalar multiplication by $-3$.

\item
$\Jac(\widetilde{F}_4)(\F_9)$ is isomorphic to $(\Z/4\Z)^{\oplus 6}$.
\end{enumerate}
\end{lem}

\begin{proof}
(1) \ $\Jac(\widetilde{F}_4)$ is isogenous to
the product of three elliptic curves over $\F_3$
because the isogeny (\ref{Jacobian:Isogeny}) is also defined over $\F_3$.
Each of them is isomorphic to either
$Y^2 = X^3 + 4 X$ or $Y^2 = X^3 - 4 X$.
The number of $\F_3$-rational points (including the point at infinity)
of each of these elliptic curves is equal to $4$.
Hence every eigenvalue of the Frobenius morphism over $\F_3$
on the $\ell$-adic Tate module of these elliptic curves
is a square root of $-3$.
Since the Frobenius eigenvalues are invariant under isogeny,
the same assertion holds for $V_{\ell} \Jac(\widetilde{F}_4)$.

(2) \ This assertion follows from (1) and the semisimplicity of
the action of the Frobenius morphism on the $\ell$-adic Tate module.

(3) \ By (2), the order of $\Jac(\widetilde{F}_4)(\F_9)$
is equal to $(1 - (-3))^6 = 4096$;
see \cite[Chapter IV, Section 21, Theorem 4]{Mumford:AbelianVarieties}.
For any $\ell \neq 3$,
a torsion point $x \in \Jac(\widetilde{F}_4)[\ell^{\infty}]$
of $\ell$-power order is defined over $\F_9$
if and only if $[-3]x = x$.
Hence we have
$\Jac(\widetilde{F}_4)(\F_9) = \Jac(\widetilde{F}_4)[4]$.
It is isomorphic to $(\Z/4\Z)^{\oplus 6}$.
\end{proof}

\begin{rem}
For an abelian variety $A$ of dimension $g$ over
the finite field $\F_q$ with $q$ elements,
the Riemann Hypothesis (also called  the Weil conjecture, proved by Weil himself)
for $A$ implies the number of $\F_q$-rational points on $A$
is less than or equal to $(1 + q^{1/2})^{2g}$;
see \cite[Chapter IV, Section 21, Theorem 4]{Mumford:AbelianVarieties}.
Lemma \ref{Lemma:MordellWeilGroup2} (3)
shows that the maximal number of elements is realized by
$\Jac(\widetilde{F}_4)$ over $\F_9$.
\end{rem}

We shall show the Mordell-Weil group $\Jac(F_4)(\Q(\zeta_8))$ is killed by $4$.
This result is presumably well-known.
One possible approach is to use the isogeny (\ref{Jacobian:Isogeny}).
But it is a non-trivial task to eliminate the possibility that
$\Jac(F_4)(\Q(\zeta_8))$ might have a non-trivial $8$-torsion point.
Instead of using (\ref{Jacobian:Isogeny}), in the following,
we shall give a proof using the reduction modulo $3$ of $\Jac(F_4)$.

\begin{prop}
\label{Proposition:KilledBy4}
$\Jac(F_4)(\Q(\zeta_8))$ is killed by $4$.
\end{prop}

\begin{proof}
(This proof was inspired by
a previous work of the second author \cite{Ito:ManinMumford}.)
Let $v$ be a finite place of $\Q(\zeta_8)$ above $3$.
The residue field $\kappa(v)$ at $v$ is $\F_9$.
The Fermat quartic $F_4$ naturally extends to
a smooth projective curve over the ring of integers of
the completion of $\Q(\zeta_8)$ at $v$.
Its special fiber is the quartic $\widetilde{F}_4$ in
Lemma \ref{Lemma:MordellWeilGroup2}.
By Lemma \ref{Lemma:MordellWeilGroup1},
every element of $\Jac(F_4)(\Q(\zeta_8))$ is of finite order.
Its order is a power of $2$. (In particular, its order is prime to $3$.)
Hence the reduction modulo $v$ homomorphism
\[
\xymatrix{
\Jac(F_4)(\Q(\zeta_8)) \ar^-{}[r] & \Jac(F_{4,v})(\F_9)
}
\]
is injective.
The group $\Jac(F_4)(\Q(\zeta_8))$ is killed by $4$
by Lemma \ref{Lemma:MordellWeilGroup2} (3).
\end{proof}

By Proposition \ref{Proposition:KilledBy4},
we have only to calculate the subgroup of $\Jac(F_4)[4]$
fixed by $\Gal(\Q(2^{1/4},\zeta_8)/\Q(\zeta_8))$.
By Theorem \ref{MainTheorem2},
it is solved by a standard method in linear algebra.

\begin{thm}
\label{MainTheorem3}
The Mordell-Weil group $\Jac(F_4)(\Q(\zeta_8))$ coincides
with the group $\mathscr{C}$ generated by divisors supported on the cusps.
In particular, the following homomorphism is an isomorphism:
\begin{align*}
  (\Z/4\Z)^{\oplus 5} \oplus \Z/2\Z &\overset{\cong}{\longrightarrow} \Jac(F_4)(\Q(\zeta_8)), \\
  (c_1,c_2,c_3,c_4,c_5,c_6) &\mapsto \sum_{i=1}^{5} c_i e_i + c_6 e_6.
\end{align*}
\end{thm}

\begin{proof}
By Proposition \ref{Proposition:KilledBy4},
we have $\Jac(F_4)(\Q(\zeta_8)) \subset \Jac(F_4)[4]$.
Since we know the Galois action on $\Jac(F_4)[4]$,
the calculation of $\Jac(F_4)(\Q(\zeta_8))$ is easy.
Since $\Gal(\Q(2^{1/4},\zeta_8)/\Q(\zeta_8))$ is generated by $\sigma^2$,
we have
\begin{align*}
\Jac(F_4)(\Q(\zeta_8))
&\cong \{\, v \in (\Z/4\Z)^{\oplus 6} \mid \rho_4(\sigma)^2 v = v \, \}.
\end{align*}
By Theorem \ref{MainTheorem2}, the matrix $\rho_4(\sigma)^2$ is
calculated as
\[
\rho_4(\sigma)^2 =
\begin{pmatrix}
1 & 0 & 0 & 0 & 0 & 0 \\
0 & 1 & 0 & 0 & 0 & 0 \\
0 & 0 & 1 & 0 & 0 & 0 \\
0 & 0 & 0 & 1 & 0 & 2 \\
0 & 0 & 0 & 0 & 1 & 0 \\
0 & 0 & 0 & 0 & 0 & 1
\end{pmatrix}.
\]
Hence
$\Jac(F_4)(\Q(\zeta_8))$ is generated by $e_1$, $e_2$, $e_3$, $e_4$, $e_5$, $e_6$.
It coincides with the group $\mathscr{C}$;
see Corollary \ref{Corollary:Rohrlich}.
\end{proof}

\section{Points over quadratic extensions of $\Q(\zeta_8)$}
\label{Section:QuadraticPoints}

In this section,
we shall determine all of the points on the Fermat quartic $F_4$
defined over quadratic extensions of $\Q(\zeta_8)$.

We shall use the following notation:
a pair of $\overline{\Q}$-rational points $P,Q \in C(\overline{\Q})$
on a curve $C$ over a number field $K$
is called a \textit{conjugate pair over $K$}
if none of $P,Q$ is defined over $K$,
and there exists a quadratic extension $L/K$
such that both of $P,Q$ are defined over $L$
and $P,Q$ are interchanged by the action of $\Gal(L/K)$.
If a pair $P,Q$ is a conjugate pair over $K$,
then the divisor $P+Q$ is defined over $K$.

The following result is essentially due to Faddeev.

\begin{lem}[Faddeev {\cite[p.1150, Section 3]{Faddeev}}]
\label{Lemma:QuadraticPoints}
Let $K$ be a number field.
Let $C \subset \PP^2$ be a smooth plane quartic defined over $K$.
Assume that the following conditions are satisfied:
\begin{enumerate}
\item $C$ has at least $N_1$ points $P_1,\ldots,P_{N_1}$ defined over $K$,
\item $C$ has at least $N_2$ conjugate pairs
$Q_1,\overline{Q}_1,\ldots,Q_{N_2},\overline{Q}_{N_2}$ over $K$
(i.e.\ for each $1 \leq i \leq N_2$,
the pair $Q_i,\overline{Q}_i$ is a conjugate pair over $K$), and
\item the number of linear equivalence classes of effective divisors of
degree $2$ on $C$ defined over $K$ is equal to $M$.
\end{enumerate}
Then the inequality
$N_1(N_1 + 1)/2 + N_2 \leq M$
is satisfied.
Moreover, if the equality
$N_1(N_1 + 1)/2 + N_2 = M$
is satisfied, then the points
$P_1,\ldots,P_{N_1}$, $Q_1$, $\overline{Q}_1$, $\ldots$, $Q_{N_2}$, $\overline{Q}_{N_2}$
are the only points on $C$ defined over quadratic extensions of $K$.
\end{lem}

\begin{proof}
We shall observe that
if $D = P + Q$ is an effective divisor of degree $2$ defined over $K$,
then either both of $P,Q$ are $K$-rational points,
or $P,Q$ is a conjugate pair over $K$.
In the setting of this lemma, $C$ has the following effective divisors
of degree $2$ defined over $K$:
\begin{align*}
P_i + P_j \quad (1 \leq i \leq j \leq N_1), & &
Q_k + \overline{Q}_k \quad (1 \leq k \leq N_2).
\end{align*}
In total, we have $N_1 (N_1 + 1)/2 + N_2$ effective divisors of degree $2$
defined over $K$.
Since $C$ is non-hyperelliptic, any two of them are not linearly equivalent
to each other.
Hence the inequality $N_1 (N_1 + 1)/2 + N_2 \leq M$ is satisfied.

Moreover, if the equality $N_1 (N_1 + 1)/2 + N_2 = M$ is satisfied,
the above divisors represent all of the linear equivalence classes of
effective divisors of degree $2$ on $C$ defined over $K$.
Therefore, there does not exist any $K$-rational point
other than $P_1,\ldots,P_{N_1}$,
and there does not exist any conjugate pair over $K$
other than $Q_1,\overline{Q}_1,\ldots,Q_{N_2},\overline{Q}_{N_2}$.
\end{proof}

Here are some points on $F_4$
defined over quadratic extensions of $\Q(\zeta_8)$.

\begin{enumerate}
\item There are $12$ cusps $A_i, B_i, C_i$ $(0 \leq i \leq 3)$.
All of them are defined over $\Q(\zeta_8)$.

\item For any $i,j$ with $0 \leq i,j \leq 3$,
the following points
\begin{align*}
[ 2^{1/4} \zeta^i_4 : \zeta^{1+2j}_8 : 1 ], & &
[ \zeta^{1+2j}_8 : 2^{1/4} \zeta^i_4 : 1 ], & &
[ 2^{-1/4} \zeta^i_4 : 2^{-1/4} \zeta^j_4 : 1 ]
\end{align*}
are defined over $\Q(2^{1/4},\zeta_8)$.
In total, we have $48$ points defined over $\Q(2^{1/4},\zeta_8)$.
Since none of them is defined over $\Q(\zeta_8)$,
we have $24$ conjugate pairs over $\Q(\zeta_8)$.

\item For any $i,j$ with $0 \leq i,j \leq 3$,
the following points
\begin{align*}
[ \zeta_3 \zeta^{1+2i}_8 : \zeta^2_3 \zeta_8^{1+2j} : 1 ], & &
[ \zeta^2_3 \zeta^{1+2i}_8 : \zeta_3 \zeta_8^{1+2j} : 1 ]
\end{align*}
are defined over $\Q(\zeta_3, \zeta_8)$.
In total, we have $32$ points defined over $\Q(\zeta_3, \zeta_8)$.
Since none of them is defined over $\Q(\zeta_8)$,
we have $16$ conjugate pairs over $\Q(\zeta_8)$.

\item We put
$\alpha := (1+\sqrt{-7})/2$ and $\overline{\alpha} := (1-\sqrt{-7})/2$.
For any $i,j$ with $0 \leq i,j \leq 3$, the following points
\begin{align*}
[ \alpha \zeta^i_4 : \overline{\alpha} \zeta^j_4 : 1 ], & &
[ \overline{\alpha} \zeta^i_4 : \alpha \zeta^j_4 : 1 ], \\
[ \overline{\alpha} \zeta^{7+2j}_8 : \zeta_4^3 : \alpha \zeta^i_4 ], & &
[ \alpha \zeta^{7+2j}_8 : \zeta_4^3 : \overline{\alpha} \zeta^i_4 ], \\
[ 1 : \alpha \zeta^{1+2i}_8 : \overline{\alpha} \zeta^{2+2j}_8 ], & &
[ 1 : \overline{\alpha} \zeta^{1+2i}_8 : \alpha \zeta^{2+2j}_8 ]
\end{align*}
are defined over $\Q(\sqrt{-7}, \zeta_8)$.
Note that the points in the second (resp.\ the third) row are
obtained from the points in the first row
by applying the automorphism $\theta_3$ (resp.\ $\theta^2_3$);
see Appendix \ref{Appendix:Automorphisms}.
In total, we have $96$ points defined over $\Q(\sqrt{-7}, \zeta_8)$.
Since none of them is defined over $\Q(\zeta_8)$,
we have $48$ conjugate pairs over $\Q(\zeta_8)$.
\end{enumerate}

\begin{rem}
The 4 points 
$[ \pm \alpha : \pm \overline{\alpha} : 1 ]$,
$[ \pm \overline{\alpha} : \pm \alpha : 1 ]$
are defined over $\Q(\sqrt{-7})$.
These points were given by Faddeev in  \cite[p.1150, Section 3]{Faddeev}.
Note that these points were already found by Aigner in 1934 \cite{Aigner}.
(See also Mordell's 1968 paper \cite{Mordell}.)
\end{rem}

\begin{thm}
\label{MainTheorem4}
The above points on the Fermat quartic $F_4$
are the only points defined over
quadratic extensions of $\Q(\zeta_8)$.
In particular, the number of points on $F_4$
defined over quadratic extensions of $\Q(\zeta_8)$ is
\[ 12 + 48 + 32 + 96 = 188. \]
\end{thm}

\begin{proof}
We have already seen that $F_4$ has
at least $12$ points defined over $\Q(\zeta_8)$.
Also, $F_4$ has at least
$88 = (48 + 32 + 96)/2$
conjugate pairs over $\Q(\zeta_8)$.
Since $12(12+1)/2 + 88 = 166$,
it is enough to show that the number of linear equivalence classes of
effective divisors of degree $2$ on $C$ defined over $\Q(\zeta_8)$
is equal to $166$.

This can be shown as follows.
Since $F_4$ has a $\Q(\zeta_8)$-rational point (such as $B_0 = [1 : 0 : 1]$),
the following map is bijective:
\[
  \Jac(F_4)(\Q(\zeta_8)) \overset{\cong}{\longrightarrow} \Pic^2(F_4 \otimes_{\Q} \Q(\zeta_8)),
  \qquad
  [D] \mapsto [D + 2 B_0].
\]

By Theorem \ref{MainTheorem3},
we know that $\Jac(F_4)(\Q(\zeta_8))$
is an abelian group of order $2048$.
We also know generators of this group explicitly.
Therefore, we can make a list of divisors representing
all of the $2048$ elements of $\Pic^2(F_4 \otimes_{\Q} \Q(\zeta_8))$.
For each divisor of degree $2$,
there is an algorithm using Gr\"obner basis
which calculates the dimension of the global sections of the line bundle
associated to it.
It is implemented in standard computer algebra systems.
Hence we can count the number of linear equivalence classes of
effective divisors of degree $2$ on $F_4$ defined over $\Q(\zeta_8)$.
It turns out that this number is equal to $166$.
(The authors calculated it using Singular.
See Appendix \ref{Appendix:MethodsCalculation} for the methods of calculation.)
\end{proof}

The following result is
considered as an analogue of Fermat's Last Theorem for
the quartic equation $X^4 + Y^4 = Z^4$ over
quadratic extensions of $\Q(\zeta_8)$.

\begin{cor}
\label{Corollary:FermatQuadraticQZeta8}
Let $K$ be a quadratic extension of $\Q(\zeta_8)$
which does not contain any of $2^{1/4}$, $\zeta_3$, $\sqrt{-7}$.
Then there do not exist any $K$-rational point
on the Fermat quartic $F_4$
other than the $12$ cusps $A_i, B_i, C_i$ $(0 \leq i \leq 3)$.
\end{cor}

\begin{rem}
The non-existence of non-cuspidal $\Q(\zeta_8)$-rational points on
the Fermat quartic $F_4$ was proved by
Klassen-Schaefer \cite[Proposition 6.1]{KlassenSchaefer}.
They used the finiteness of $\Jac(F_4)(\Q(\zeta_8))$
and Coleman's theory of $p$-adic abelian integrals (for $p=5$).
Their proof and our proof are completely different.
\end{rem}

\appendix

\section{The Mordell-Weil group over subfields of $\Q(\zeta_8)$}
\label{Appendix:MordellWeilGroupSubfields}

The $8$-th cyclotomic field $\Q(\zeta_8)$ contains
the following $4$ proper subfields:
$\Q$, $\Q(\sqrt{-1})$, $\Q(\sqrt{2})$, $\Q(\sqrt{-2})$.
In this appendix, we shall calculate
the Mordell-Weil group of $\Jac(F_4)$ over each subfield of $\Q(\zeta_8)$.

\subsection{The Mordell-Weil group over $\Q$}

In order to illustrate our methods, we shall explain how to
calculate $\Jac(F_4)(\Q)$ using our results in this paper.
Since
$\Gal(\Q(2^{1/4},\zeta_8)/\Q)$ is generated by $\sigma$ and $\tau$,
we have
\begin{align*}
\Jac(F_4)(\Q)
&\cong \{\, v \in (\Z/4\Z)^{\oplus 6} \mid (\rho_4(\sigma) - I_6)v = (\rho_4(\tau) - I_6)v = 0 \,\},
\end{align*}
where $I_6$ denotes the identity matrix of size $6$.

The calculation of the right hand side is
an easy exercise in linear algebra.
We shall consider the following invertible matrix mod $4$:
\[
P :=
\begin{pmatrix}
1 & 0 & 1 & 0 & 0 & 3 \\
0 & 1 & 0 & 0 & 0 & 0 \\
0 & 0 & 1 & 0 & 0 & 0 \\
0 & 0 & 0 & 1 & 0 & 0 \\
2 & 0 & 0 & 0 & 3 & 2 \\
0 & 0 & 2 & 0 & 3 & 3
\end{pmatrix}
\]
(This matrix $P$ was found by elementary column operations
applied to both matrices $\rho_4(\sigma) - I_6$ and $\rho_4(\tau) - I_6$
simultaneously.)
Then the matrices
$(\rho_4(\sigma) - I_6) P$, $(\rho_4(\tau) - I_6) P$
are equal to
\begin{align*}
\begin{pmatrix}
0 & 0 & 2 & 0 & 1 & 3 \\
0 & 0 & 2 & 0 & 3 & 1 \\
0 & 0 & 2 & 0 & 1 & 1 \\
2 & 0 & 2 & 0 & 2 & 1 \\
0 & 0 & 0 & 0 & 0 & 2 \\
0 & 0 & 0 & 0 & 2 & 2
\end{pmatrix}, & & 
\begin{pmatrix}
2 & 0 & 0 & 0 & 1 & 2 \\
0 & 0 & 0 & 0 & 0 & 1 \\
0 & 0 & 0 & 0 & 1 & 0 \\
1 & 0 & 0 & 0 & 0 & 0 \\
0 & 0 & 0 & 0 & 2 & 0 \\
0 & 0 & 0 & 0 & 0 & 2
\end{pmatrix},
\end{align*}
respectively.
We put
$P^{-1} v = (x_1,x_2,x_3,x_4,x_5,x_6)^{\mathrm{T}}$
for some $x_1,\ldots,x_7 \in \Z/4\Z$,
where the superscript $\mathrm{T}$ denotes the transpose.
Assume that
$(\rho_4(\sigma) - I_6) v = (\rho_4(\tau) - I_6) v = 0$.
Then, from the fourth row of $(\rho_4(\tau) - I_6)P$,
we have $x_1 = 0$.
Similarly, from the third and the second rows of $(\rho_4(\tau) - I_6)P$,
we have $x_5 = x_6 = 0$.
Then, from the first row of $(\rho_4(\sigma) - I_6)P$,
we have $2 x_3 = 0$. Hence we have $x_3 \in 2\Z/4\Z$.
The elements $x_2,x_4$ are arbitrary because all of the entries in
the second and the fourth columns of
$(\rho_4(\sigma) - I_6) P$, $(\rho_4(\tau) - I_6) P$
are zero.

From these calculations, we see that
$\Jac(F_4)(\Q)$ is an abelian group of order $32$ isomorphic to
$(\Z/4\Z)^{\oplus 2} \oplus \Z/2\Z$.
Moreover, looking at the second, third, and the fourth columns of $P$,
we see that the following homomorphism is an isomorphism:
\begin{align*}
  (\Z/4\Z)^{\oplus 2} \oplus \Z/2\Z &\overset{\cong}{\longrightarrow}
   \Jac(F_4)(\Q), \\
  (c_1,c_2,c_3) &\mapsto c_1 e_2 + c_2 e_4 + c_3 (2 e_1 + 2 e_3).
\end{align*}

Since $F_4$ has a $\Q$-rational point (such as $B_0 = [1 : 0 : 1]$),
every element of $\Jac(F_4)(\Q)$ is represented by a divisor defined over $\Q$;
see \cite[Chapter 8, Proposition 4]{BoschLuetkebohmertRaynaud},
\cite[Section 5.7.1]{Poonen:RationalPoints}.
Let us confirm it for the basis $e_2, e_4, 2 e_1 + 2 e_3$ of $\Jac(F_4)(\Q)$.
By definition, $e_2, e_4$ are represented by divisors defined over $\Q$.
It is a non-trivial fact that the divisor class
$2 e_1 + 2 e_3 = [2 A_1 + 2 B_1 - 4 B_0]$
is represented by a divisor defined over $\Q$.
(The points $A_1, B_1$ are not $\Q$-rational points.)
To see this, we note that the divisor $2 A_1 + 2 B_1 - 4 B_0$
is linearly equivalent to $A_0 + B_0 - A_2 - B_2$.
(This can be checked using Rohrlich's results;
see Section \ref{Section:Rohrlich}.)
Hence we have
$2 e_1 + 2 e_3 = [A_0 + B_0 - A_2 - B_2]$.
Since $A_0,A_2,B_0,B_2$ are $\Q$-rational points,
we see that $2 e_1 + 2 e_3$ is represented by a divisor defined over $\Q$.

\begin{rem}
The structure of $\Jac(F_4)(\Q)$ is presumably well-known.
In \cite{Faddeev}, Faddeev claimed (without proof) that $\Jac(F_4)(\Q)$ has order $32$.
In \cite[p.19, Proposition 2]{Klassen:Thesis},
Klassen describes an isomorphism 
$\Jac(F_4)(\Q) \cong (\Z/4\Z)^{\oplus 2} \oplus \Z/2\Z$,
which is essentially the same as above.
Klassen also gave a sketch of the proof in \cite[p.18-20]{Klassen:Thesis},
but his proof does not seem complete.
The authors of this paper could not find a complete proof of
this isomorphism in the literature.
\end{rem}

\subsection{The Mordell-Weil group over $\Q(\sqrt{-1})$}

The method of the calculation is basically the same as in the case of $\Q$.

The Galois group 
$\Gal(\Q(2^{1/4},\zeta_8)/\Q(\sqrt{-1}))$ is
a cyclic group of order $4$ generated by $\sigma$.
Hence we have
\begin{align*}
\Jac(F_4)(\Q(\sqrt{-1}))
&\cong \{\, v \in (\Z/4\Z)^{\oplus 6} \mid (\rho_4(\sigma) - I_6) v = 0 \,\},
\end{align*}

Since we have
\[
(\rho_4(\sigma) - I_6) \begin{pmatrix}
1 & 0 & 0 & 0 & 0 & 0 \\
0 & 1 & 0 & 0 & 0 & 0 \\
0 & 0 & 1 & 0 & 0 & 0 \\
0 & 0 & 0 & 1 & 0 & 0 \\
0 & 0 & 0 & 0 & 1 & 1 \\
0 & 0 & 0 & 0 & 0 & 3
\end{pmatrix}
=
\begin{pmatrix}
0 & 0 & 0 & 0 & 2 & 1 \\
0 & 0 & 0 & 0 & 2 & 3 \\
0 & 0 & 0 & 0 & 0 & 1 \\
0 & 0 & 0 & 0 & 1 & 0 \\
0 & 0 & 0 & 0 & 2 & 0 \\
0 & 0 & 0 & 0 & 0 & 2
\end{pmatrix},
\]
the Mordell-Weil group
$\Jac(F_4)(\Q(\sqrt{-1}))$ is isomorphic to $(\Z/4\Z)^{\oplus 4}$.
The following homomorphism is an isomorphism:
\begin{align*}
  (\Z/4\Z)^{\oplus 4} &\overset{\cong}{\longrightarrow}
   \Jac(F_4)(\Q(\sqrt{-1})), \\
  (c_1,c_2,c_3,c_4) &\mapsto c_1 e_1 + c_2 e_2 + c_3 e_3 + c_4 e_4.
\end{align*}
By definition, the divisor classes $e_1, e_2, e_3, e_4$
are represented by divisors defined over $\Q(\sqrt{-1})$.

\subsection{The Mordell-Weil group over $\Q(\sqrt{2})$}

We have
\begin{align*}
\Jac(F_4)(\Q(\sqrt{2}))
&\cong \{\, v \in (\Z/4\Z)^{\oplus 6} \mid (\rho_4(\sigma)^2 - I_6)v = (\rho_4(\tau) - I_6)v = 0 \,\}
\end{align*}
because $\Gal(\Q(2^{1/4},\zeta_8)/\Q(\sqrt{2}))$
is generated by $\sigma^2$ and $\tau$.
For the square matrix
\[
P :=
\begin{pmatrix}
1 & 0 & 1 & 0 & 0 & 3 \\
0 & 1 & 0 & 0 & 0 & 0 \\
0 & 0 & 1 & 0 & 0 & 0 \\
0 & 0 & 0 & 1 & 0 & 0 \\
2 & 0 & 0 & 0 & 3 & 2 \\
0 & 0 & 2 & 0 & 3 & 3
\end{pmatrix},
\]
the matrices $(\rho_4(\sigma)^2 - I_6) P$, $(\rho_4(\tau) - I_6) P$
are equal to
\begin{align*}
\begin{pmatrix}
0 & 0 & 0 & 0 & 0 & 0 \\
0 & 0 & 0 & 0 & 0 & 0 \\
0 & 0 & 0 & 0 & 0 & 0 \\
0 & 0 & 0 & 0 & 2 & 2 \\
0 & 0 & 0 & 0 & 0 & 0 \\
0 & 0 & 0 & 0 & 0 & 0
\end{pmatrix}, & &
\begin{pmatrix}
2 & 0 & 0 & 0 & 1 & 2 \\
0 & 0 & 0 & 0 & 0 & 1 \\
0 & 0 & 0 & 0 & 1 & 0 \\
1 & 0 & 0 & 0 & 0 & 0 \\
0 & 0 & 0 & 0 & 2 & 0 \\
0 & 0 & 0 & 0 & 0 & 2
\end{pmatrix},
\end{align*}
respectively.
Therefore, $\Jac(F_4)(\Q(\sqrt{2}))$ is
isomorphic to $(\Z/4\Z)^{\oplus 3}$.
The following homomorphism is an isomorphism:
\begin{align*}
  (\Z/4\Z)^{\oplus 3} &\overset{\cong}{\longrightarrow}
   \Jac(F_4)(\Q(\sqrt{2})), \\
  (c_1,c_2,c_3) &\mapsto c_1 e_2 + c_2 (e_1 + e_3 + 2 e'_6) + c_3 e_4.
\end{align*}

Let us give divisors over $\Q(\sqrt{2})$ which give
a basis of $\Jac(F_4)(\Q(\sqrt{2}))$.
The divisor classes $e_2, e_4$ are defined over $\Q$.
Hence they are defined over $\Q(\sqrt{2})$.
It can be checked (with the aid of computer algebra systems)
that the divisor class $e_1 + e_3 + 2 e'_6$ is represented by
the divisor
\[
[ -2^{-1/4} : 2^{-1/4} : 1 ] + [ 2^{-1/4} : -2^{-1/4} : 1 ] + A_0 - 3 B_0.
\]
This divisor is defined over $\Q(\sqrt{2})$.

\subsection{The Mordell-Weil group over $\Q(\sqrt{-2})$}

We have
\begin{align*}
& \Jac(F_4)(\Q(\sqrt{-2})) \\
&\cong \{\, v \in (\Z/4\Z)^{\oplus 6} \mid (\rho_4(\sigma)^2 - I_6)v
               = (\rho_4(\tau \sigma) - I_6)v = 0 \,\}
\end{align*}
because $\Gal(\Q(2^{1/4},\zeta_8)/\Q(\sqrt{-2}))$
is generated by $\sigma^2$ and $\tau \sigma$.

For the square matrix
\[
P :=
\begin{pmatrix}
3 & 0 & 3 & 0 & 0 & 2 \\
0 & 1 & 0 & 0 & 0 & 0 \\
0 & 0 & 1 & 0 & 0 & 0 \\
0 & 0 & 0 & 1 & 0 & 0 \\
2 & 0 & 2 & 0 & 0 & 1 \\
0 & 0 & 2 & 0 & 1 & 3
\end{pmatrix},
\]
the matrices
$(\rho_4(\sigma)^2 - I_6) P$,
$(\rho_4(\tau \sigma) - I_6) P$
are equal to
\begin{align*}
\begin{pmatrix}
0 & 0 & 0 & 0 & 0 & 0 \\
0 & 0 & 0 & 0 & 0 & 0 \\
0 & 0 & 0 & 0 & 0 & 0 \\
0 & 0 & 0 & 0 & 2 & 2 \\
0 & 0 & 0 & 0 & 0 & 0 \\
0 & 0 & 0 & 0 & 0 & 0
\end{pmatrix}, & &
\begin{pmatrix}
2 & 0 & 0 & 0 & 3 & 0 \\
0 & 0 & 0 & 0 & 0 & 1 \\
0 & 0 & 0 & 0 & 1 & 0 \\
1 & 0 & 0 & 0 & 0 & 0 \\
0 & 0 & 0 & 0 & 2 & 2 \\
0 & 0 & 0 & 0 & 0 & 2
\end{pmatrix},
\end{align*}
respectively.
Therefore,
$\Jac(F_4)(\Q(\sqrt{-2}))$ is isomorphic to $(\Z/4\Z)^{\oplus 3}$.
The following homomorphism is an isomorphism:
\begin{align*}
  (\Z/4\Z)^{\oplus 3} &\overset{\cong}{\longrightarrow}
   \Jac(F_4)(\Q(\sqrt{-2})), \\
  (c_1,c_2,c_3) &\mapsto c_1 e_2 + c_2 (3 e_1 + e_3 + 2 e_5 + 2 e'_6) + c_3 e_4.
\end{align*}

Let us give divisors over $\Q(\sqrt{-2})$ which give
a basis of $\Jac(F_4)(\Q(\sqrt{-2}))$.
The divisor classes $e_2, e_4$ are defined over $\Q$.
Hence they are defined over $\Q(\sqrt{-2})$.
It can be checked (with the aid of computer algebra systems)
that the divisor class $3 e_1 + e_3 + 2 e_5 + 2 e'_6$
is represented by the divisor
\begin{align*}
B_2 + C_2 + C_3 - 3 A_0 &= 
[-1 : 0 : 1] + [\zeta_8^5 : 1 : 0] + [\zeta_8^7 : 1 : 0] - 3 [0 : 1 : 1].
\end{align*}
This divisor is defined over $\Q(\sqrt{-2})$
since $\zeta_8^5,\zeta_8^7$ are the roots of
the quadratic polynomial $X^2 + \sqrt{-2} X - 1$ over $\Q(\sqrt{-2})$.

\section{Explicit calculation of the Weil pairing}
\label{Appendix:WeilPairing}

Recall that we have a non-degenerate alternating bilinear form
\begin{align*}
\langle,\rangle \colon \Jac(F_4)[4] \times \Jac(F_4)[4]
&\longrightarrow \mu_4 \cong \Z/4\Z,
\end{align*}
called the {\it Weil pairing},
where $\mu_4 := \{\, \pm 1, \pm \zeta_4 \,\}$ is the group of
fourth roots of unity.
The isomorphism $\mu_4 \cong \Z/4\Z$ is given by
$\zeta_4 \mapsto 1$.

Since the Weil pairing $\langle,\rangle$ is Galois equivariant,
the image of the mod $4$ Galois representation $\rho_4$
sits in the symplectic similitude group
\begin{align*}
  &  \mathrm{GSp}(\Jac(F_4)[4], \langle,\rangle) \\
  &:= \bigg\{\, (g,c) \in \Aut(\Jac(F_4)[4]) \times (\Z/4\Z)^{\times}
      \ \bigg| \ 
      \begin{array}{l} \langle g \alpha, g \beta \rangle = c \langle \alpha, \beta \rangle \\
      \text{for all}\ \alpha, \beta \in \Jac(F_4)[4]
      \end{array}
      \bigg\}.
\end{align*}

In order to calculate the image of $\rho_4$ inside
$\mathrm{GSp}(\Jac(F_4)[4], \langle,\rangle)$,
we need to calculate the matrix
\[
\begin{pmatrix}
\langle e_1,e_1 \rangle  & \langle e_1,e_2 \rangle  & \langle e_1,e_3 \rangle &
\langle e_1,e_4 \rangle  & \langle e_1,e_5 \rangle  & \langle e_1,e'_6 \rangle \\
\langle e_2,e_1 \rangle  & \langle e_2,e_2 \rangle  & \langle e_2,e_3 \rangle &
\langle e_2,e_4 \rangle  & \langle e_2,e_5 \rangle  & \langle e_2,e'_6 \rangle \\
\langle e_3,e_1 \rangle  & \langle e_3,e_2 \rangle  & \langle e_3,e_3 \rangle &
\langle e_3,e_4 \rangle  & \langle e_3,e_5 \rangle  & \langle e_3,e'_6 \rangle \\
\langle e_4,e_1 \rangle  & \langle e_4,e_2 \rangle  & \langle e_4,e_3 \rangle &
\langle e_4,e_4 \rangle  & \langle e_4,e_5 \rangle  & \langle e_4,e'_6 \rangle \\
\langle e_5,e_1 \rangle  & \langle e_5,e_2 \rangle  & \langle e_5,e_3 \rangle &
\langle e_5,e_4 \rangle  & \langle e_5,e_5 \rangle  & \langle e_5,e'_6 \rangle \\
\langle e'_6,e_1 \rangle & \langle e'_6,e_2 \rangle & \langle e'_6,e_3 \rangle &
\langle e'_6,e_4 \rangle & \langle e'_6,e_5 \rangle & \langle e'_6,e'_6 \rangle
\end{pmatrix}
\]
representing the Weil pairing
with respect to the basis $e_1,e_2,e_3,e_4,e_5,e'_6$.

Here is the summary of our results on the explicit calculation of the Weil pairing.

\begin{thm}
\label{Theorem:WeilPairing}
With respect to the basis $e_1,e_2,e_3,e_4,e_5,e'_6$ of $\Jac(F_4)[4]$,
the matrix representing the Weil pairing $\langle,\rangle$ is given by
\[
\begin{pmatrix}
0 & 1 & 3 & 0 & 3 & 1 \\
3 & 0 & 3 & 0 & 1 & 1 \\
1 & 1 & 0 & 2 & 3 & 0 \\
0 & 0 & 2 & 0 & 0 & 1 \\
1 & 3 & 1 & 0 & 0 & 1 \\
3 & 3 & 0 & 3 & 3 & 0
\end{pmatrix}.
\]
\end{thm}

\begin{proof}
First, we recall how to calculate the Weil pairing
on a smooth projective curve $C$ over a field $k$.
Let $n \geq 1$ be a positive integer invertible in $k$,
and $D,E$ divisors of degree $0$ on $C$ defined over $\overline{k}$
representing $n$-torsion points on $\Jac(F_4)$.
Take non-zero rational functions $f_D,f_E$ satisfying
$\mathrm{div}\,f_D = n D$ and $\mathrm{div}\,f_E = n E$.
Then the Weil pairing is calculated by the following formula:
\[ \langle [D],[E] \rangle = \prod_{P \in C(\overline{k})}
   (-1)^{n (\ord_P D)(\ord_P E)} \frac{f_E^{\ord_P D}}{f_D^{\ord_P E}}(P). \]
(See \cite[Theorem 1]{Howe:WeilPairing} and references therein.)

We shall apply the above formula for the Fermat quartic $F_4$ and $n = 4$.
Let $\delta$ be the element as in the proof of
Proposition \ref{Proposition:KeyDivisor}.
We define $c'_1, c'_2, c'_3, c'_4, c'_5$ by
\begin{align*}
c'_1 &:= -52 \delta^7+9 \delta^6+224 \delta^5-36 \delta^4-496 \delta^3+63 \delta^2+452 \delta+93, \\
c'_2 &:= 148 \delta^7+80 \delta^6-476 \delta^5-280 \delta^4+736 \delta^3+449 \delta^2+88 \delta-58, \\
c'_3 &:= 216 \delta^7+71 \delta^6-780 \delta^5-253 \delta^4+1476 \delta^3+515 \delta^2-312 \delta-160, \\
c'_4 &:= 73 \delta^6-365 \delta^4+730 \delta^2-365, \\
c'_5 &:= -80 \delta^7+158 \delta^6+352 \delta^5-553 \delta^4-824 \delta^3+728 \delta^2+568 \delta+203.
\end{align*}
We define linear polynomials $f_1,f_2,f_3,f_4,f_5,g_1$ by
\begin{align*}
f_1 &:= Y - \zeta_4 Z, &
f_2 &:= Y + Z, &
f_3 &:= X - \zeta_4 Z, \\
f_4 &:= X + Z, &
f_5 &:= X - \zeta_8^3 Y, &
g_1 &:= X - Z.
\end{align*}
Moreover, we define a cubic polynomial $f'_6$ by
\begin{align*}
f'_6 &:=
219 (X^3 + Y^3) + c'_1 (X^2 Y + X Y^2) + c'_2 (X^2 Z + Y^2 Z) \\
& + c'_3 (X Z^2 + Y Z^2) + c'_4 Z^3 + c'_5 XYZ.
\end{align*}

Then, it can be checked (with the aid of computer algebra systems)
that the divisors associated with the rational functions
$f_1/g_1$, $f_2/g_1$, $f_3/g_1$, $f_4/g_1$, $f_5/g_1$, $f'_6/g^3_1$
are
\begin{align*}
\mathrm{div}(f_1/g_1) &= 4 (A_1 - B_0), &
\mathrm{div}(f_2/g_1) &= 4 (A_2 - B_0), \\
\mathrm{div}(f_3/g_1) &= 4 (B_1 - B_0), &
\mathrm{div}(f_4/g_1) &= 4 (B_2 - B_0), \\
\mathrm{div}(f_5/g_1) &= 4 (C_1 - B_0), &
\mathrm{div}(f'_6/g^3_1) &= 4 (P_1 + P_2 + P_3 - 3 B_0).
\end{align*}

Therefore, we can use these rational functions to calculate the Weil pairings explicitly.
(The authors calculated them using Singular.)
\end{proof}

\begin{rem}
There is a sign issue on the Weil pairing.
In some literature, some authors seem to use the Weil pairing with
opposite signs.
In Theorem \ref{Theorem:WeilPairing},
we used the formula of the Weil pairing as presented in
\cite[Theorem 1]{Howe:WeilPairing}.
If we calculate the Weil pairing with opposite sign convention,
the matrix in Theorem \ref{Theorem:WeilPairing} is replaced
by the transpose of it.
\end{rem}

By Theorem \ref{Theorem:WeilPairing},
it is an easy exercise in linear algebra to
determine the image of the mod $4$ Galois representation $\rho_4$
inside the symplectic similitude group $\mathrm{GSp}_6(\Z/4\Z)$.
Here $\mathrm{GSp}_6(\Z/4\Z)$ is defined by
\begin{align*}
  \mathrm{GSp}_6(\Z/4\Z)
  &:= \bigg\{\, (g,c) \in \GL_6(\Z/4\Z) \times (\Z/4\Z)^{\times}
      \ \bigg| \ g^{\mathrm{T}} J g = c J
      \,\bigg\},
\end{align*}
where $g^{\mathrm{T}}$ denotes the transpose of $g$.
Here $J \in \GL_6(\Z/4\Z)$ is defined by
\[
J := \begin{pmatrix}
0 & 0 & 0 & 0 & 0 & 1 \\
0 & 0 & 0 & 0 & 1 & 0 \\
0 & 0 & 0 & 1 & 0 & 0 \\
0 & 0 & 3 & 0 & 0 & 0 \\
0 & 3 & 0 & 0 & 0 & 0 \\
3 & 0 & 0 & 0 & 0 & 0
\end{pmatrix}.
\]
(Note that $3 \equiv -1 \pmod{4}$.)
We define $e''_1,e''_2,e''_3,e''_4,e''_5,e''_6 \in \Jac(F_4)[4]$ by
\begin{align*}
e''_1 &:= e_1, &
e''_2 &:= e_2 + 3 e'_6, \\
e''_3 &:= e_3 + e'_6, &
e''_4 &:= 3 e_4 + 3 e_5 + 3 e'_6, \\
e''_5 &:= e_3 + 2 e_4 + e_5 + 2 e'_6, &
e''_6 &:= 3 e_2 + e_3 + 2 e_4 + 3 e_5 + 2 e'_6.
\end{align*}

Then we have the following corollary.

\begin{cor}
\label{Corollary:WeilPairingGSp(6)}
The elements $e''_1,e''_2,e''_3,e''_4,e''_5,e''_6$ form a basis of $\Jac(F_4)[4]$
such that the Weil pairing $\langle,\rangle$ is represented by $J$.
Moreover, the images of $\sigma, \tau \in \Gal(\Q(2^{1/4},\zeta_8)/\Q)$
by the mod $4$ Galois representation
\[ \Gal(\Q(2^{1/4},\zeta_8)/\Q) \to \mathrm{GSp}_6(\Z/4\Z)  \]
with respect to the basis $e''_1,e''_2,e''_3,e''_4,e''_5,e''_6$
are given by the following matrices:
\begin{align*}
& \begin{pmatrix}
1 & 3 & 1 & 1 & 0 & 0 \\
0 & 2 & 3 & 3 & 0 & 0 \\
0 & 0 & 1 & 3 & 1 & 3 \\
0 & 3 & 1 & 3 & 3 & 1 \\
0 & 1 & 3 & 2 & 2 & 3 \\
0 & 0 & 0 & 0 & 0 & 1
\end{pmatrix}, & &
\begin{pmatrix}
1 & 0 & 0 & 1 & 3 & 1 \\
2 & 1 & 2 & 1 & 1 & 3 \\
3 & 1 & 3 & 1 & 1 & 1 \\
3 & 1 & 2 & 1 & 2 & 0 \\
3 & 0 & 1 & 3 & 3 & 0 \\
0 & 3 & 3 & 1 & 2 & 3
\end{pmatrix}.
\end{align*}
\end{cor}

\begin{proof}
The assertions immediately follow from
Theorem \ref{MainTheorem2} and Theorem \ref{Theorem:WeilPairing}.
\end{proof}

\section{The action of automorphisms on $\Jac(F_4)[4]$}
\label{Appendix:Automorphisms}

It is classically known that
the automorphism group of the Fermat quartic $F_4 \otimes_{\Q} \C$ over $\C$
is a non-abelian group of order $96$.
All of the automorphisms of $F_4 \otimes_{\Q} \C$ are defined over $\Q(\zeta_8)$.
The automorphism group $\Aut(F_4 \otimes_{\Q} \Q(\zeta_8))$
is generated by the following elements (see \cite[p.173, Theorem]{Tzermias:Automorphism}):
\begin{align*}
\theta_1 \colon [X : Y : Z] &\mapsto [\zeta_4 X : Y : Z], \\
\theta_2 \colon [X : Y : Z] &\mapsto [Y : X : Z], \\
\theta_3 \colon [X : Y : Z] &\mapsto [\zeta_8^7 Y : \zeta_4^3 Z : X].
\end{align*}

\begin{rem}
\label{Remark:AutomorphismTheta3}
The following explanation might be useful to understand
the nature 
the automorphism group of $F_4$.
Over the $8$-th cyclotomic field $\Q(\zeta_8)$,
the Fermat quartic $F_4$ is isomorphic to another quartic
\[ F'_4 := \{\, [X:Y:Z] \in \PP^2 \mid X^4 + Y^4 + Z^4 = 0 \,\} \]
via the isomorphism:
\[
\psi \colon F_4 \otimes_{\Q} \Q(\zeta_8)
\overset{\cong}{\to} F'_4 \otimes_{\Q} \Q(\zeta_8), \qquad
[X : Y : Z] \mapsto [X : Y : \zeta_8^7 Z].
\]
Under this isomorphism $\psi$,
the automorphism $\theta_3$ is translated into the cyclic permutation
of the coordinates $X,Y,Z$. Precisely, we have
\begin{align*}
\psi^{-1} \theta_3 \psi \colon [X : Y : Z] &\mapsto [Y : Z : X].
\end{align*}
\end{rem}

It is easy to see that $\theta_2,\theta_3$ generate a subgroup of
$\Aut(F_4 \otimes_{\Q} \Q(\zeta_8))$
isomorphic to the symmetric group $\mathfrak{S}_3$ of order $6$.
Since $\theta_2^2 = \id$,
we see that $\theta_2 \theta_{1} \theta_2$ is conjugate to $\theta_1$.
We have
\begin{align*}
\theta_2 \theta_1 \theta_2 \colon [X : Y : Z] &\mapsto [X : \zeta_4 Y : Z].
\end{align*}
From this, we see that
$\theta_1$ and $\theta_2 \theta_1 \theta_2$ generate an abelian subgroup of
$\Aut(F_4 \otimes_{\Q} \Q(\zeta_8))$
isomorphic to $(\Z/4 \Z)^{\oplus 2}$.
Therefore, $\Aut(F_4 \otimes_{\Q} \Q(\zeta_8))$
is isomorphic to the semi-direct product
$(\Z/4 \Z)^{\oplus 2} \rtimes \mathfrak{S}_3$.
The action of the automorphisms on the $4$-torsion points
gives a homomorphism
\[ \Aut(F_4 \otimes_{\Q} \Q(\zeta_8)) \to \Aut(\Jac(F_4)[4]) \cong \GL_6(\Z/4\Z). \]

We shall calculate this homomorphism explicitly as follows.

\begin{thm}
\label{Theorem:Automorphisms}
With respect to the basis $e_1,e_2,e_3,e_4,e_5,e'_6 \in \Jac(F_4)[4]$,
the actions of the automorphisms $\theta_{1},\theta_{2},\theta_{3}$
are represented by the following matrices:
\begin{align*}
\begin{pmatrix}
1 & 0 & 0 & 0 & 3 & 1 \\
0 & 1 & 0 & 0 & 1 & 1 \\
3 & 3 & 3 & 2 & 2 & 1 \\
0 & 0 & 1 & 3 & 1 & 2 \\
0 & 0 & 0 & 0 & 3 & 1 \\
0 & 0 & 0 & 0 & 2 & 3
\end{pmatrix}, & &
\begin{pmatrix}
2 & 2 & 3 & 2 & 1 & 2 \\
3 & 3 & 3 & 0 & 0 & 1 \\
3 & 2 & 2 & 2 & 1 & 2 \\
3 & 0 & 3 & 3 & 0 & 1 \\
0 & 0 & 0 & 0 & 3 & 0 \\
0 & 0 & 0 & 0 & 2 & 1
\end{pmatrix}, & &
\begin{pmatrix}
3 & 2 & 3 & 0 & 3 & 1 \\
2 & 3 & 3 & 2 & 2 & 2 \\
2 & 1 & 2 & 2 & 2 & 3 \\
1 & 2 & 1 & 1 & 2 & 3 \\
1 & 3 & 0 & 0 & 0 & 1 \\
0 & 2 & 0 & 0 & 0 & 3
\end{pmatrix}.
\end{align*}
\end{thm}

\begin{proof}
The actions of $\theta_1, \theta_2, \theta_3$
on the cusps $A_i, B_i, C_i \ (0 \leq i \leq 3)$
are calculated as follows:
\[
\begin{array}{|c||c|c|c|c||c|c|c|c||c|c|c|c|}
\hline
\quad & A_0 & A_1 & A_2 & A_3 & B_0 & B_1 & B_2 & B_3 & C_0 & C_1 & C_2 & C_3 \\
\hline
\theta_1 & A_0 & A_1 & A_2 & A_3 & B_1 & B_2 & B_3 & B_0 & C_1 & C_2 & C_3 & C_0 \\
\hline
\theta_2 & B_0 & B_1 & B_2 & B_3 & A_0 & A_1 & A_2 & A_3 & C_3 & C_2 & C_1 & C_0 \\
\hline
\theta_3 & C_0 & C_1 & C_2 & C_3 & A_3 & A_2 & A_1 & A_0 & B_3 & B_2 & B_1 & B_0 \\
\hline
\end{array}
\]

The actions of $\theta_1, \theta_2, \theta_3$
on the points $P_1, P_2, P_3$
are calculated as follows:
\[
\begin{array}{|c||c|c|c|}
\hline
\quad & P_1 & P_2 & P_3 \\
\hline
\theta_1 & [-2^{1/4} : \zeta_8 : 1] & [\zeta_8^3 : 2^{1/4} \zeta_4 : 1] &
[-2^{-1/4} \zeta_4 : -2^{-1/4} : 1] \\
\hline
\theta_2 & P_2 & P_1 & P_3 \\
\hline
\theta_3 & [2^{-1/4} \zeta^3_4 : -2^{-1/4} : 1] & [2^{1/4} : \zeta_8^5 : 1] &
[\zeta_8^7 : 2^{1/4} \zeta^3_4 : 1] \\
\hline
\end{array}
\]

Therefore,
with the aid of computer algebra systems,
the actions of $\theta_1, \theta_2, \theta_3$
on the elements $e_1, e_2, e_3, e_4, e_5, e'_6$
are calculated as follows:
\[
\begin{array}{|c||c|c|}
\hline
\quad & e_1 & e_2 \\
\hline
\theta_1 &
[A_1 - B_1] = e_1 + 3e_3 &
[A_2 - B_1] = e_2 + 3e_3 \\
\hline
\theta_2 &
\begin{array}{c} [B_1 - A_0] = \\ 2e_1 + 3e_2 + 3e_3 + 3e_4 \end{array} &
[B_2 - A_0] = 2e_1 + 3e_2 + 2e_3 \\
\hline
\theta_3 &
\begin{array}{c} [C_1 - A_3] = \\ 3e_1 + 2e_2 + 2e_3 + e_4 + e_5 \end{array} &
\begin{array}{c} [C_2 - A_3] = \\ 2e_1 + 3e_2 +  e_3 + 2e_4 + 3e_5 + 2e'_6 \end{array} \\
\hline
\end{array}
\]

\[
\begin{array}{|c||c|c|}
\hline
\quad & e_3 & e_4 \\
\hline
\theta_1 &
[B_2 - B_1] = 3e_3 + e_4 &
[B_3 - B_1] = 2 e_3 + 3e_4 \\
\hline
\theta_2 &
[A_1 - A_0] = 3e_1 + 3e_2 + 2e_3 + 3e_4 &
[A_2 - A_0] = 2e_1 + 2e_3 + 3e_4 \\
\hline
\theta_3 &
[A_2 - A_3] = 3e_1 + 3e_2 + 2e_3 + e_4 &
[A_1 - A_3] = 2e_2 + 2e_3 + e_4 \\
\hline
\end{array}
\]

\[
\begin{array}{|c||c|c|}
\hline
\quad & e_5 \\
\hline
\theta_1 &
[C_2 - B_1] = 3e_1 + e_2 + 2e_3 + e_4 + 3e_5 + 2e'_6 \\
\hline
\theta_2 & [C_2 - A_0] = e_1 + e_3 + 3e_5 + 2e'_6 \\
\hline
\theta_3 &
[B_2 - A_3] = 3e_1 + 2e_2 + 2e_3 + 2e_4 \\
\hline
\end{array}
\]

\[
\begin{array}{|c||c|}
\hline
\quad & e'_6 \\
\hline
\theta_1
 &
\begin{array}{c}
[[-2^{1/4} : \zeta_8 : 1] + [\zeta_8^3 : 2^{1/4} \zeta_4 : 1] +
[-2^{-1/4} \zeta_4 : -2^{-1/4} : 1] - 3 B_1] \\
= 2e_2 + e_3 + e_5 + e'_6
\end{array} \\
\hline
\theta_2
 & 
[P_2 + P_1 + P_3 - 3 A_0] = 2e_1 + e_2 + 2e_3 + e_4 + e'_6 \\
\hline
\theta_3
 & 
\begin{array}{c}
[[2^{-1/4} \zeta^3_4 : -2^{-1/4} : 1] + [2^{1/4} : \zeta_8^5 : 1] +
[\zeta_8^7 : 2^{1/4} \zeta^3_4 : 1] - 3 A_3]\\
= e_1 + 2e_2 + 3e_3 + 3e_4 + e_5 + 3e'_6
\end{array} \\
\hline
\end{array}
\]
\end{proof}

\section{How to find the element $e'_6$}
\label{Appendix:Experimental}

The $4$-torsion point $e'_6$ (defined in Section \ref{Section:Notation})
plays an important role in this paper.
The authors found it experimentally.
Here are the methods of the authors to find this element.

Since every divisor of degree $3$ on $F_4$ is linearly equivalent to
an effective divisor,
the authors tried to find a divisor $D$ of degree $3$ such that
$[D - 3 B_0]$ gives a $4$-torsion point on $\Jac(F_4)$
which does not belong to the subgroup $\mathscr{C}$
generated by divisors supported on the cusps.

The authors first studied possible candidates of the field of definition of
such a divisor $D$. (In fact, the following strategy is inspired by
recent developments on Iwasawa theory and extensions between
mod $p$ Galois representations.)
Let $e'_6$ be a $4$-torsion point on $\Jac(F_4)$
which does not belong to $\mathscr{C}$.
Let $K$ be the smallest extension of $\Q(\zeta_8)$ where $e'_6$ is defined.
Since the $4$-torsion points $\Jac(F_4)[4]$
are generated by $\mathscr{C}$ and $e'_6$,
the extension $K/\Q$ corresponds to the kernel of the mod $4$ Galois
representation
\[ \rho_4 \colon \Gal(\overline{\Q}/\Q) \to \Aut(\Jac(F_4)[4]) \cong \GL_6(\Z/4\Z). \]
In particular, we see that $K/\Q$ is a Galois extension.

Since all of the $2$-torsion points on $\Jac(F_4)$ are defined
over $\Q(\zeta_8)$,
for each $s \in \Gal(K/\Q(\zeta_8))$,
the difference $s(e'_6) - e'_6$ is killed by $2$.
Hence we have an injective homomorphism
\[ \psi \colon \Gal(K/\Q(\zeta_8)) \to \Jac(F_4)[2], \qquad
   s \mapsto s(e'_6) - e'_6. \]
We see that
the Galois group $\Gal(K/\Q(\zeta_8))$ is isomorphic to an elementary abelian group
of type $(2,\ldots,2)$.

Since $\Jac(F_4)$ has good reduction outside $2$,
the extension $K/\Q$ is unramified outside $2$.
Moreover, $K$ must satisfy the following condition:
every finite place of $\Q(\zeta_8)$ above $3$ splits in $K/\Q(\zeta_8)$.
(Here is a proof of this claim.
Let $v$ be a finite place of $\Q(\zeta_8)$ above $3$,
and take a finite place $w$ of $K$ above $v$.
Since $\psi$ is injective, the Frobenius element $\mathrm{Frob}_w$ at $w$
has non-zero image under $\psi$.
Hence $\mathrm{Frob}_w$ acts non-trivially on the reduction of $e'_6$ modulo $w$.
It contradicts the fact that the residue field at $v$ is $\F_9$,
and all of the $4$-torsion points on the reduction modulo $3$ of $\Jac(F_4)$
are defined over $\F_9$; see Lemma \ref{Lemma:MordellWeilGroup2} (3).)

The authors \textit{guessed} $K/\Q(\zeta_8)$ is a quadratic extension,
and tried to find candidates of $K$.
Using Sage, it is not difficult to check that $\Q(\zeta_8)$ has exactly
$7$ quadratic extensions which are unramified outside $2$.
Among them, exactly $3$ quadratic extensions are Galois over $\Q$.
Looking at the finite places above $3$,
it turned out that $\Q(2^{1/4},\zeta_8)$ is the only quadratic extension
of $\Q(\zeta_8)$ satisfying all of the above conditions.
Then, the authors \textit{guessed} $K$ is $\Q(2^{1/4},\zeta_8)$,
and tried to find a divisor of degree $3$ defined over
$\Q(2^{1/4},\zeta_8)$.
It was not difficult to find
the $\Q(2^{1/4},\zeta_8)$-rational points $P_1,P_2,P_3$.
(See Section \ref{Section:Notation} for the definition of these points.)
Calculating the divisor class of $P_1 + P_2 + P_3 - 3 B_0$ using Singular,
the authors finally found the divisor class
$e'_6 := [P_1 + P_2 + P_3 - 3 B_0]$
is what they were looking for.

\section{Methods of calculation}
\label{Appendix:MethodsCalculation}

In this appendix,
we give some remarks on the methods of calculation.

The key computational results in this paper are
the identity $2 e'_6 = 2e_2 + 2e_4 + e_6$
in Proposition \ref{Proposition:KeyDivisor},
and the calculation of the number of effective divisor classes
in $\Pic^2(F_4 \otimes_{\Q} \Q(\zeta_8))$
in the proof of Theorem \ref{MainTheorem4}.

Here is a sample source code for Singular which
performs necessary calculation.
{\small
\begin{verbatim}
LIB "divisors.lib";
ring r=(0,d),(x,y,z),dp; minpoly = d^8 - 4*d^6 + 8*d^4 - 4*d^2 + 1;
ideal I = x^4 + y^4 - z^4; qring Q = std(I);
number z8 = (2*d^6 - 7*d^4 + 11*d^2 - 1)/3; number z4 = z8^2;
number a = (d^7 - 5*d^5 + 10*d^3 - 8*d)/3;
divisor A0 = makeDivisor(ideal(x, y-z),    ideal(1));
divisor A1 = makeDivisor(ideal(x, y-z4*z), ideal(1));
divisor A2 = makeDivisor(ideal(x, y+z),    ideal(1));
divisor A3 = makeDivisor(ideal(x, y+z4*z), ideal(1));
divisor B0 = makeDivisor(ideal(x-z,    y), ideal(1));
divisor B1 = makeDivisor(ideal(x-z4*z, y), ideal(1));
divisor B2 = makeDivisor(ideal(x+z,    y), ideal(1));
divisor B3 = makeDivisor(ideal(x+z4*z, y), ideal(1));
divisor C0 = makeDivisor(ideal(x-z8*y,    z), ideal(1));
divisor C1 = makeDivisor(ideal(x-z8*z4*y, z), ideal(1));
divisor C2 = makeDivisor(ideal(x+z8*y,    z), ideal(1));
divisor C3 = makeDivisor(ideal(x+z8*z4*y, z), ideal(1));
divisor E1 = A1 + multdivisor(-1, B0);
divisor E2 = A2 + multdivisor(-1, B0);
divisor E3 = B1 + multdivisor(-1, B0);
divisor E4 = B2 + multdivisor(-1, B0);
divisor E5 = C1 + multdivisor(-1, B0);
divisor E6 = A1 + B1 + C1 + A2 + B2 + C2 + multdivisor(-6, B0);
divisor P1 = makeDivisor(ideal(x-a*z4*z, y-z8*z), ideal(1));
divisor P2 = makeDivisor(ideal(x-z8*z, y-a*z4*z), ideal(1));
divisor P3 = makeDivisor(ideal(x-a^(-1)*z, y-a^(-1)*z), ideal(1));
divisor EE6 = P1 + P2 + P3 + multdivisor(-3, B0);
print("Is 2e'_6 equal to 2e_2 + 2e_4 + e_6?");
if(linearlyEquivalent(multdivisor(2, EE6), multdivisor(2, E2)
  + multdivisor(2, E4) + E6)[1] != 0)
  { print("Yes"); }else{ print("No"); }
divisor D1 = multdivisor(2, P1 + P2 + P3) + multdivisor(-1, A1)
  + multdivisor(-3, A2) + multdivisor(4, B0) + multdivisor(-1, B1)
  + multdivisor(-3, B2) + multdivisor(-1, C1 + C2);
number c1 = d^6 - 2*d^4 + d^2 + 7;
number c2 = 2*d^6 - 10*d^4 + 20*d^2 - 13;
number c3 = 22*d^7 + 9*d^6 - 86*d^5 - 36*d^4 + 166*d^3 + 63*d^2 - 86*d - 24;
number c4 = 10*d^7 - 2*d^6 - 26*d^5 + 7*d^4 + 22*d^3 - 20*d^2 + 46*d - 14;
number c5 = 22*d^6 - 77*d^4 + 154*d^2 - 44;
divisor D2 = makeDivisor(
  ideal( 33*(x^3 + x*y^2 + x*y*z - x*z^2 - y^2*z - y*z^2)
  + c3*(-x^2*y + x*y*z) + c4*(x^2*z + x*y*z - x*z^2 - y*z^2)
  + c5*(x*z^2 - z^3)),
  ideal(3*(x^3 + y^3 + z^3)
  + c1*(x^2*y + x^2*z + x*y^2 + x*y*z + y^2*z - z^3)
  - c2*(x*y*z + x*z^2 + y*z^2 + z^3)));
print("Proof of Proposition 4.1: calculation of div(f)");
if(isEqualDivisor(D1,D2)==1){ print("OK"); }else{ print("Not OK"); }
print("The number of effective elements in Pic^2(F_4)(Q(zeta_8))");
int c,p1,p2,p3,p4,p5,p6; c=0;
for(p1=0;p1<=3;p1++){ for(p2=0;p2<=3;p2++){ for(p3=0;p3<=3;p3++){
for(p4=0;p4<=3;p4++){ for(p5=0;p5<=3;p5++){ for(p6=0;p6<=1;p6++){
if(size(globalSections(multdivisor(p1, E1) + multdivisor(p2, E2)
  + multdivisor(p3, E3) + multdivisor(p4, E4) + multdivisor(p5, E5)
  + multdivisor(p6, E6) + multdivisor(2, B0))[1]) == 1){ c=c+1; }}}}}}};
print(c);
\end{verbatim}
}

Generally speaking, the calculation of divisor classes over
a number field of large degree (such as $\Gal(\Q(2^{1/4},\zeta_8)$)
is very slow.
Here is a simple trick to reduce the amount of computer calculation:
in order to confirm identities between $4$-torsion points on $\Jac(F_4)$,
it is enough to confirm them over any finite field of characteristic different from $2$.
In fact, the authors took the prime number $73$,
which is the smallest prime number which splits completely in
$\Q(2^{1/4},\zeta_8)$.
They confirmed identities in this paper between the $4$-torsion points
by calculating them over $\F_{73}$.
The calculation over $\F_{73}$ is much faster
because $\F_{73}$ is a prime field.
The authors chose $10 \pmod{73}$ as a primitive $8$-th root of unity in $\F_{73}$,
and $18 \pmod{73}$ as a fourth root of $2$ in $\F_{73}$
because they satisfy the relation $10 + 10^7 \equiv 18^2 \pmod{73}$.

\section*{Acknowledgments}

The authors would like to thank Takuya Yamauchi
for asking a question on the calculation of the Galois image
inside the symplectic similitude group $\mathrm{GSp}_6(\Z/4\Z)$.
Most of calculations, especially the calculations of linear equivalences of
divisors, were done with the aid of the computer algebra systems Maxima \cite{Maxima},
Sage \cite{Sage}, and Singular \cite{Singular}.
Especially, Singular's library \texttt{divisors.lib}
\cite{Singular:DivisorsLib} helped the authors very much.

The work of the first author was supported by
JSPS KAKENHI Grant Number 13J01450 and 16K17572.
The work of the second author 
was supported by JSPS KAKENHI Grant Number 20674001 and 26800013.
The work of the third author
was supported by JSPS KAKENHI Grant Number 26800011 and 18H05233.
This work was supported by the Sumitomo Foundation FY2018 Grant
for Basic Science Research Projects (Grant Number 180044).

\end{document}